\documentclass[12pt,a4paper]{amsart}

\usepackage{amscd}
\usepackage{graphicx}
\usepackage{amssymb,amsfonts,latexsym}

\newif\iffurther
\furtherfalse

\input xy
\xyoption {all}

\usepackage{geometry}
\newtheorem{thm}{Theorem}[section]
\newtheorem{cor}[thm]{Corollary}
\newtheorem{lem}[thm]{Lemma}
\newtheorem{prop}[thm]{Proposition}
\newtheorem{exam}[thm]{Example}
\newtheorem{ques}[thm]{Question}
\newtheorem{prob}[thm]{Problem}
\theoremstyle{definition}

\newtheorem{defn}[thm]{Definition}
\theoremstyle{remark}
\newtheorem{rem}[thm]{Remark}
\numberwithin{equation}{section}

\newcommand\Tref[1]{{Theorem~\ref{#1}}}
\newcommand\Pref[1]{{Proposition~\ref{#1}}}
\newcommand\Lref[1]{{Lemma~\ref{#1}}}
\newcommand\Cref[1]{{Corollary~\ref{#1}}}
\newcommand\Rref[1]{{Remark~\ref{#1}}}
\newcommand\Dref[1]{{Definition~\ref{#1}}}
\newcommand\Sref[1]{{Section~\ref{#1}}}
\newcommand\Ssref[1]{{Subsection~\ref{#1}}}
\newcommand\Qref[1]{{Question~\ref{#1}}}
\newcommand\Eref[1]{{Example~\ref{#1}}}
\newcommand\Fref[1]{{Figure~\ref{#1}}}
\newcommand\eq[1]{{(\ref{#1})}}

\long\def\forget#1\forgotten{}

\newcommand{\isom}{\cong}

\newcommand{\M}{\mathcal{M}}
\newcommand{\fp}{\frak{p}}
\newcommand{\mQ}{\mathbb{Q}}

\newcommand{\Z}{\mathbb{Z}}
\newcommand{\F}{\mathbb{F}}
\def\s{{\sigma}}

\newcommand{\set}[1]{\left\{#1\right\}}

\newcommand{\ra}{\rightarrow}
\newcommand\sg[1]{{\left<{#1}\right>}}
\newcommand\subjectto{{\,|\ }}
\newcommand\suchthat{{\,:\ }}

\DeclareMathOperator\Gal{Gal}
\def\({\left(}
\def\){\right)}
\def\divides{{\,|\,}}
\def\co{{\,{:}\,}}
\def\ndivides{{\not|\,}}
\newcommand\card[1]{{\left|{#1}\right|}}
\newcommand\dimcol[2]{{[{#1}\!:\!{#2}]}}

\newcommand\intpart[1] {{\lceil #1 \rceil}}
\newcommand\oline[1] {{\overline{#1}}}

\newcommand\Br{{\operatorname{Br}}}

\newcommand\tame[1]{{{#1}_{\operatorname{tr}}}}
\newcommand\unram[1]{{{#1}_{\operatorname{un}}}}

\newcommand\res[1][{}] {{\operatorname{res}_{#1}}}
\newcommand\sep[1]{{\widetilde{#1}}}

\newcommand\Aut{{\operatorname{Aut}}}

\newcommand\sub{{\,\subset\,}}
\renewcommand\Im{{\operatorname{Im}}}

\newcommand\SL[1][]{{\operatorname{SL}_{#1}}}

\newcommand\PGL[1][]{{\operatorname{PGL}_{#1}}}
\newcommand\vspan{{\operatorname{span}}}
\newcommand\rank{{\operatorname{rank}}}
\newcommand\odd{{\operatorname{odd}}}
\newcommand\mul[1]{{#1^\times}}
\newcommand\md[2][p]{{\mul{#2}/\mul{#2}^{#1}}}
\newcommand\pro[3][]{{\if#1N{(#2)^{\!\operatorname{N}}_{#3}}\else{(#2)_{#3}}\fi}} 

\def\Caseiv{{(\ref{caseiv})}}     
\def\Caseviii{{(\ref{caseviii})}} 

\renewcommand\nu{{v}}
\renewcommand\varpi{{w}}

\def\sensitive{sensitive}

\begin{document}

\title[Realizability and admissibility under extensions] 
{Realizability and admissibility under extension of $p$-adic and number fields}

\def\Tech{Department of Mathematics, Technion-Israel Institute
of Technology, Haifa 32000, Israel}

\author{ Danny Neftin}

\address{\Tech}
\email{neftind@tx.technion.ac.il}%

\author{ Uzi Vishne }

\def\BIU{Deptartment of Mathematics, Bar Ilan University,
 Ramat Gan 52900, Israel}
\address{\BIU}
\email{vishne@math.biu.ac.il}




\begin{abstract}
\forget 
A finite group $G$ is $K$-admissible if there exists a
$G$-crossed product $K$-division algebra. In this manuscript we
study the behavior of admissibility under extensions of number
fields $M/K$. Assuming the  Grunwald-Neukirch property over $M$,
our study reduces to realization problems over $p$-adic fields. We
treat the latter for groups $G$ for which every Sylow subgroup is
normal, or metacyclic and any $2$-Sylow subgroup is metacyclic,
and for all extensions of $p$-adic fields, $p$ odd, but a list of
$17$ sensitive extensions. For such groups that are $K$-admissible,
we give necessary and sufficient conditions for the group to be
$M$-admissible whenever $M/K$ locally avoids the sensitive
extensions. 
\forgotten
%
A finite group $G$ is $K$-admissible if there exists a $G$-crossed
product $K$-division algebra. In this manuscript we study the
behavior of admissibility under extensions of number fields $M/K$.
We show that in many cases, including Sylow metacyclic and nilpotent groups whose order is prime to the number of roots of unity in $M$,  a
$K$-admissible group $G$ is $M$-admissible if and only if $G$
satisfies the easily verifiable Liedahl condition over $M$. 

\end{abstract}

\date{Sep. 21, 2010}%

\maketitle


%


\section{Introduction}

Let $K$ be a field. A field $L\supseteq K$ is called $K$-adequate
if it is contained as a maximal subfield in a finite dimensional
central $K$-division algebra. A group $G$ is $K$-admissible if
there is a $G$-extension $L/K$, i.e. $L/K$ is a Galois extension
with $\Gal(L/K) \cong G$, so that $L$ is $K$-adequate.
Equivalently, $G$ is $K$-admissible if there is a $G$-crossed
product $K$-division algebra.
Ever since adequacy and admissibility were introduced in
\cite{Sch}, they were studied extensively over various
types of fields, especially over number fields. 

As oppose to realizability of groups as Galois groups, there are
known restrictions on the number fields $K$ over which a given group
is $K$-admissible. Liedahl's condition (which was shown by
Schacher \cite{Sch} over $\mQ$, and generalized by Liedahl \cite[Theorem~28]{Lid2}) describes such a restriction.
We say that $G$ satisfies Liedahl's condition over $K$, if for
every prime $p$ dividing $\card{G}$, one of the following holds:
\begin{enumerate}
\item[(i)] $p$ decomposes in $K$ (has at least two prime divisors),
\item[(ii)] $p$ does not decompose in $K$, and a $p$-Sylow subgroup $G(p)$ of
$G$ is metacyclic and admits a Liedahl presentation over $K$ (for
details see Definition \ref{def:LC} which is based on
\cite{Lid2}).
\end{enumerate}








In \cite[Theorem~9.1]{Sch}, Schacher showed that any finite group
$G$ is admissible over some number field $K$. However, for many
groups $G$ it is an open problem to determine the number fields
over which they are admissible. In fact, searching for an explicit
description for all groups seems hopeless. 

In this paper we fix a field $K$ over which $G$ is admissible and
ask over which finite extensions of $K$, $G$ is still admissible.
By assuming our group $G$ is realizable over $M$ and furthermore
can be realized over $M$ with prescribed local conditions, i.e.
satisfying the Grunwald-Neukirch (GN) property, this question
reduces to the following local realization problem:

\begin{prob}\label{lrp}
Let $m/k$ be an extension of $p$-adic fields and $G$ a group that
is realizable over $k$. Is there a subgroup $H$ of $G$ which
is realizable over $m$ and contains a $p$-Sylow subgroup of $G$?
\end{prob}

At first we consider the case of $p$-groups, where the problem is whether a $p$-group that is realizable over $k$, is realizable over an extension $m$ of $k$.
For $p$ odd, we
notice that the maximal pro-$p$ quotient $\oline{G_k(p)}$ of the
absolute Galois group $G_k$ is covered by $\oline{G_m(p)}$, providing a positive answer: 

\begin{prop}\label{prop 1}
Let $m/k$ be a finite extension of $p$-adic fields where $p$ is an
odd prime. Then any $p$-group that is realizable over $k$ is also
realizable over $m$.
\end{prop}

The simplest behavior one can hope for in terms of admissibility,
is that a $K$-admissible group $G$ would be $M$-admissible if and
only if it satisfies Liedahl's condition over $M$. This is indeed
the case for various classes of groups:
\begin{enumerate}
\item When all Sylow subgroups of $G$ are cyclic \cite[Theorem~2.8]{Sch};
\item When $G$ is abelian and does not fall into a special case over
$M$ \cite{Chr};
\item For metacyclic groups \cite{Lid2},\cite{Lidw};
\item For $G = \SL[2](5)$ \cite{FF};
\item For $G = A_6$ or $G = A_7$ \cite{ScS};
\item For $G = \PGL[2](7)$ \cite{AS}, and
\item For the Symmetric groups $G=S_n$, $1\leq n\leq 17$, $n\not=12,13$
(by \cite{GL}, \cite{Lid2} and \cite{Sch}).
\end{enumerate}

{}Using \Pref{prop 1}, we are able to add all the odd-order
$p$-groups having the GN-property over $M$ to this list:
\begin{prop}\label{intro-admiss of 2-gp}
Let $M/K$ be an extension of number fields and $p$ an odd prime.
Let $G$  be a $p$-group that is $K$-admissible and has the
GN-property over $M$. Then $G$ is $M$-admissible if and only if $G$
satisfies Liedahl's condition over $M$.
\end{prop}

Propositions \ref{prop 1} and \ref{intro-admiss of 2-gp} are
proved in \Ssref{ss:podd}. We note (in Section \ref{sec2}) that \Pref{intro-admiss of 2-gp} extends to nilpotent groups of odd order and (by \Rref{sylowmetacyclic.rem}) to Sylow metacyclic groups (having metacyclic Sylow subgroups).
However, the following example shows that  Problem~\ref{lrp} can have a negative answer for some $2$-groups:
\begin{exam}\label{Ex1.5}
There is a group $G$, of order $2^{6}$, which is realizable over
$\mQ_2$ but not over $\mQ_2(\sqrt{-1})$.
\end{exam}
For a proof, see \Cref{G is not realizable over Q_2(i)}.
In \Pref{3.8} we interpret this example globally:
\begin{exam}
There is a rational prime $q$ for which the group $G$ of
\Eref{Ex1.5} is $\mQ(\sqrt{q})$-admissible,  satisfies
Liedahl's condition over $\mQ(\sqrt{-1},\sqrt{q})$  but is not
$\mQ(\sqrt{-1},\sqrt{q})$-admissible.
\end{exam}

Liedahl showed that a similar phenomena happens for the groups
$S_{n}$, $n=12,13$ and the local extension
$\mQ_2(\sqrt{-3})/\mQ_2$ (see \cite{GL}). We shall restrict our
discussion to groups which are either of odd order, or with
metacyclic $2$-Sylow subgroups. 

For some $p$-adic extensions $m/k$, for $p$ odd, including
extensions in which the inertia degree $f(m/k)$ is a $p$-power and
$\dimcol{m}{k}>5$ we show that $G_m$ covers the maximal quotient
of $G_k$ with a normal $p$-Sylow subgroup.

We use this method to answer Problem~\ref{lrp} positively for odd primes, under the following assumptions. The list of `sensitive' extensions of $p$-adic fields,
($16$ with $p=3$ and one for $p=5$) is described in
Subsection~\ref{ss:sc}.


\begin{thm}\label{intro-local main thm}
Let $p$ be an odd prime. Let $m/k$ be a non-sensitive extension of
$p$-adic fields and $G$ a group with a normal $p$-Sylow subgroup, $P$.
Assume $G$ is realizable over $k$. Then there is a subgroup $H\leq
G$ that contains $P$ and is realizable over $m$.

\end{thm}
The question as to whether the non-sensitivity assumption can be removed
remains open. However, the assumption that every Sylow subgroup is
normal is essential: 
\begin{exam}\label{7x29}
Let  $G=C_7 \wr D$ where $$D = \sg{a,b\subjectto a^{7}=b^{29}=1,
a^{-1}ba=b^7};$$ thus the $7$-Sylow subgroups of $G$ are neither
normal nor metacyclic.

In \Eref{local-counterexample.exm} we show there exists an
extension $m/k$ of $7$-adic fields such that $G$ is realizable
over $k$, although no subgroup of $G$ that contains a $7$-Sylow
subgroup is realizable over $m$.
\end{exam}

We say that an
extension of number fields $M/K$ is sensitive if it has a
sensitive completion. The main theorem follows from \Tref{intro-local main thm}
by combining the local data (see \Ssref{ss:5.1}):
\begin{thm}\label{intro_main.thm}
Let $M/K$ be a non-sensitive extension of number fields. Let $G$ be a
group for which every Sylow subgroup is either normal or
metacyclic, and the $2$-Sylow subgroups are metacyclic.

Assume $G$ is $K$-admissible and has the GN-property over $M$.
Then $G$ is $M$-admissible if and only if $G$ satisfies Liedahl's
condition over $M$.
\end{thm}

Let $\mu_n$ denote the set of $n$-th roots of unity. As a consequence of \Tref{intro_main.thm} and \cite[Corollary
2]{Neu} we have:
\begin{cor}
Let $G$ be an odd order group for which every Sylow subgroup is
either normal or metacyclic. Let $M/K$ be a non-sensitive extension of number
fields so that $G$ is $K$-admissible and $\mu_{\card{G}}\cap M=\{1\}$. Then $G$ is $M$-admissible if and only
if $G$ satisfies Liedahl's condition over $M$.

In particular, if every prime dividing $\card{G}$ decomposes in
$M$ or if $M\cap \mQ(\mu_{\card{G}})=K\cap \mQ(\mu_{\card{G}})$,
then $G$ is $M$-admissible. \end{cor}

We also show that the assumption that every Sylow subgroup is
either normal or metacyclic is essential in \Tref{intro_main.thm}:
\begin{exam}
Let $G$ be the group defined in Example~\ref{7x29}.  In \Eref{global-counterexample},
we show furthermore that there is an extension of number fields
$M/K$ so that $G$ is $K$-admissible, satisfies Liedahl's condition
over $M$, has the GN-property over $M$, but is not $M$-admissible.
\end{exam}

As an example we use \Tref{intro_main.thm} to understand the
behavior of admissibility for a specific group (see
\Eref{mod.exm}):
\begin{exam}\label{F6.}
Let $K=\mQ(\sqrt{14})$ and $G=C_{13} \wr M_{3^3}$, where $M_{3^3}$
is the modular group
$$\sg{x,y| x^{-1}yx=y^{4}, x^3=y^{9}=1},$$
and $\wr$ is the standard wreath product. In \Eref{mod.exm}, we
show $G$ is $K$-admissible and deduce from \Tref{intro_main.thm}
that for a number field $M\supseteq K$, $G$ is $M$-admissible if
and only if $G$ satisfies Liedahl's condition. We therefore deduce
the admissibility behavior in  ~\Fref{F6} by checking Liedahl's
condition. 
\end{exam}

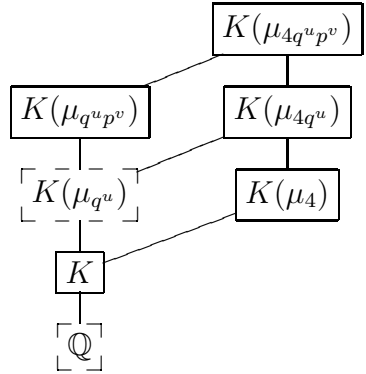
\begin{figure}\label{F6}
\begin{equation*}
\xymatrix@R=12pt{
{} & *+[F]{K(\mu_{4q^up^v})} \ar@{-}[d] \\
 *+[F]{K(\mu_{q^up^v})} \ar@{-}[d] \ar@{-}[ru] & *+[F]{K(\mu_{4q^u})}\ar@{-}[d] \\
 *+[F--]{K(\mu_{q^u})}  \ar@{-}[d] \ar@{-}[ru] & *+[F]{K(\mu_4)} \ar@{-}[ld] \\
*+[F]{K}  \ar@{-}[d] & {} \\
*+[F--]{\mQ} & {}
}
\end{equation*}
\caption{Admissibility of $C_{13}\wr M_{3^3}$: the group is admissible over the solid-boxed fields, but not over the dash-boxed ones (see \Eref{F6.}; here $u \geq 2, v\geq 1, p=13,q=3$ and $\mu_n$ are the $n$-th roots of unity.)}
\end{figure}

Similar examples (also given in  \Eref{mod.exm}) show that the
rank (the minimal number of generators) of the $p$-Sylow subgroups
of $K$-admissible groups is not bounded (as apposed to the case of
admissible $p$-groups discussed in \cite[Section 10]{Sch}).




The basic facts about admissibility of groups over number fields
are reviewed in \Sref{sec:pre}. We also discuss the behavior of
wild and tame admissibility under extension of number fields and
the connection between these types of admissibility to parts (i)
and (ii) (respectively) in Liedahl's condition. 

\medskip

The authors are greatful to Jack Sonn for his valuable advice and useful comments.

\section{Preliminaries}\label{sec:pre}




%
%
%
%
%
%
%
%
%
%

\subsection{Admissibility and Preadmissibility}\label{ss:pread}

For a prime $\nu$ of a
field $K$, we denote by $K_{\nu}$ the completion of $K$ with
respect to $\nu$. If $L/K$ is a finite Galois extension,
$L_{\nu}$ denotes the completion of $L$ with respect to some
prime divisor of $\nu$ in $L$.

The basic criterion for admissibility over global fields is
due to Schacher:

\begin{thm}[{Schacher, \cite{Sch}}]\label{Schachers criterion}
Let $L/K$ be a finite Galois extension of global fields. Then $L$
is $K$-adequate  if for every rational prime $p$ dividing $\card{G}$,
where $G = \Gal(L/K)$, there is a pair of primes $\nu_1 ,\nu_2$ of
$K$ such that each of $\Gal(L_{\nu_i}/K_{\nu_i})$ contains a
$p$-Sylow subgroup of $G$.
\forget Let $K$ be a global field and $G$ a finite group.
Then $G$ is $K$-admissible if and only if there exists a
Galois $G$-extension $L/K$ such that for every rational
prime $p\divides \card{G}$, there is a pair of primes
$\nu_1 ,\nu_2$ of $K$ such that each of $\Gal(L_{\nu_i}/K_{\nu_i})$
contains a $p$-Sylow subgroup of $G$.
\forgotten
\end{thm}

%
%
%
%

Extracting the necessary local conditions for $K$-admissibility from \Tref{Schachers criterion}, we arrive at the following definition. For a group $G$, $G(p)$ denotes a $p$-Sylow subgroup.
\begin{defn}\label{preadm.def}
Let $K$ be a number field.
The group $G$ is $K$-preadmissible if $G$ is realizable over $K$, and there exists a finite set $S = \set{\nu_i(p) \suchthat p\divides \card{G}, i = 1,2}$ of primes of $K$, and, for each $\nu \in S$, a subgroup $G^\nu \leq G$, such that
\begin{enumerate}
\item $\nu_1(p) \neq \nu_2(p)$,
\item $G^{\nu_i(p)} \supseteq G(p)$ for every $p$ and $i = 1,2$, and
\item \label{preadm.def3}
$G^\nu$ is realizable over $K_\nu$ for every $\nu \in S$.
\end{enumerate}
(Notice that a $p$-group $G$ is $K$-preadmissible if and only if there is a pair of primes $\nu_1$ and $\nu_2$ of $K$, such that $G$ is realizable over $K_{\nu_1}$ and over $K_{\nu_2}$.)
\end{defn}
Clearly, every $K$-admissible group is also $K$-preadmissible.
However the opposite does not always hold (see \cite[Example
2.14]{Nef}).

For an extension of fields $L/K$, $\Br(L/K)$ denotes the kernel of
the restriction map $\res \co \Br(K) \ra \Br(L)$. For number
fields we have the following isomorphism of groups, where $\Pi_K$ is
the set of places of $K$:
\begin{equation*}
\Br(L/K)\cong \(\bigoplus_{\pi\in\Pi_K}
\frac{1}{\gcd_{\pi' | \pi} \dimcol{L_{\pi'}}{K_{\pi}}}\mathbb{Z} /
\mathbb{Z}\)_0,\end{equation*} where $(\,\cdot\,)_0$ denotes that
the sum of invariants is zero.

Over a number field $K$, the exponent of a division algebra is
equal to its degree, and so $L$ is $K$-adequate if and only if there is an
element of order $\dimcol{L}{K}$ in $\Br(L/K)$ \cite[Proposition 2.1]{Sch}.
\subsection{Tame and wild admissibility}\label{ss:tw}

We denote by $\unram{k}$ the maximal unramified extension of a local field $k$, and by $\tame{k}$ the maximal tamely ramified extension.

The tamely ramified subgroup $\tame{\Br(L/K)}$ of $\Br(L/K)$ is
the subgroup of algebras which are split by the tamely ramified part of every
completion of $L$; namely the subgroup corresponding under the above isomorphism to
\begin{equation}
\(\bigoplus_{\pi\in\Pi_K} \frac{1}{\gcd_{\pi'|\pi} \dimcol{L_{\pi'}\cap
\tame{(K_{\pi})}}{K_{\pi}}}\Z / \Z\)_0.\end{equation}

 Following the above local description of adequacy we define:
\begin{defn}
We say that a finite extension $L$ of $K$ is \emph{tamely
$K$-adequate} if there is an element of order $\dimcol{L}{K}$ in
$\tame{\Br(L/K)}$.

{}Likewise, a finite group $G$ is \emph{tamely $K$-admissible}  if
there is a tamely $K$-adequate Galois $G$-extension $L/K$.
\end{defn}

The structure of tamely admissible groups is related to Liedahl presentations:
For $t$ prime to $n$, let $\sigma_{t,n}$ be the automorphism of
$\mQ(\mu_n)/\mQ$ defined by $\sigma_{t,n}(\zeta)=\zeta^t$ for
$\zeta\in\mu_n$.
\begin{defn}[{\cite{Lid2}}]\label{def:LC}
We say that a metacyclic $p$-group has a
\emph{Liedahl presentation} over $K$, if it has a
presentation
\begin{equation}\label{Mdef}
\M(m,n,i,t) := \sg{x,y \subjectto x^m = y^i,\ y^n = 1,\
x^{-1}yx = y^t}
\end{equation}  such that $\s_{t,n}$ fixes $K \cap
\mQ(\mu_n)$.
\end{defn}

\begin{exam}\label{D4}
The dihedral group $D_4$ has a Liedahl presentation over $\mQ$,
but not over $\mQ(\sqrt{-1})$. Thus $D_4$ satisfies the Liedahl
condition over $\mQ$, but not over $\mQ(\sqrt{-1})$.
\end{exam}

The existence of Liedahl's presentation for a $p$-group $G$ over
$K$ implies $G$ is $K$-tame-preadmissible (namely,
\Dref{preadm.def} holds with realizability within $\tame{(K_\nu)}$
in \ref{preadm.def}.(\ref{preadm.def3})\,).
\begin{rem}[{Liedahl, follows directly from \cite[Proofs of theorems~28 and~29]{Lid2}}]\label{interpert_lid.rem}
Let $G$ be a finite group. If $G$ is realizable  
over infinitely many completions of $K$ (at infinitely many
primes), then $G$ has a presentation as above. If $G$ is a $p$-group then the converse also holds. 
In addition a $p$-group is realizable over infinitely many completions of $K$ if and only if it is realizable over a
completion $K_v$ at one prime $v$ that does not divide $p$.
\end{rem}

This allows us to simplify the definition of preadmissibility:
\begin{lem} \label{pread_support.lem}
Let $K$ be a number field. A group $G$ is $K$-preadmissible if and
only if it is realizable over $K$, and there are distinct primes
$v_i(p)$, $p$ runs over the prime dividing $\card{G}$ and $i =
1,2$, such that for every $p$ and $i = 1,2$, there is a subgroup
$H\leq G$ that contains a $p$-Sylow subgroup of $G$ and is
realizable over $K_{v_i(p)}$.
\end{lem}
\begin{proof}
The if part holds by definition.
To prove the only if part let $$T=\set{v_i(p) \subjectto i=1,2, p\divides \card{G} }$$ be a set of primes of $K$ and for every prime $v\in T$  a corresponding subgroup $G^v$ so that:

 1) $v_1(p)\not=v_2(p)$,

 2) $G^v$ is realizable over over $K_v$,

 3) $G^{v_i(p)}$ contains a $p$-Sylow subgroup of $G$,

for every $i=1,2$, $p\divides \card{G}$.  We shall define primes $w_i(p)$, $i=1,2$,$p\divides \card{G}$ such that all primes are distinct and for every $w_i(p)$ there is a subgroup of $G$ that contains a $p$-Sylow subgroup of $G$ and is realizable over $M_{w_i(p)}$.

If $v_i(p),$ divides $p$ define $w_i(p)=v_i(p)$ for any $i=1,2,p\divides \card{G}$. If $v_i(p)$ does not divide $p$ then $G(p)$ is metacyclic and has a Liedahl presentation over $K$ (by \Rref{interpert_lid.rem}).  Thus, there are infinitely many primes $w$ of $M$ for which $G(p)$ is realizable over $M_w$. For all primes $v_i(p)$  that do not divide $p$ (running over both $i$ and $p$) choose distinct primes $w_i(p)$ which are not in $T$ and for which $G(p)$ is realizable over $M_{w_i(p)}$ (such a choice is possible since there are infinitely many such $w$'s). We have chose distinct primes $w_i(p), i=1,2,p\divides\card{G}$ as required.
\end{proof}

\begin{rem}\label{section0 - remark on down-inheritance of Lid cond}
If a $p$-group $G$ has a Liedahl presentation over $M$, then $G$
also has a Liedahl presentation over any subfield $K$ of $M$. 
\end{rem}

\forget
\begin{rem}\label{0.7}
Let $k$ be a local field. Then $\Gal(\tame{k}/k)$ is the
pro-finite completion of $$\sg{x,y \subjectto x^{-1} y x = y^{q}}$$
where $q = \card{\bar{k}}$.
\end{rem}
\forgotten

\begin{thm}[{Liedahl \cite{Lid2}, see also \cite{Nef}}]\label{admissibility implies Liedahl} 
If $G$ is tamely $K$-admissible, then $G(p)$ has a Liedahl
presentation over $K$ for every prime $p$ dividing $\card{G}$.
\end{thm}
 There are no known counterexamples to the opposite
implication. However, the following two results are proved for
$p$-groups in \cite[Theorem~30]{Lid2} and in general in
\cite{Nef}: 

\begin{thm}\label{theorem tame admissibility of solvable groups}
Let $K$ be a number field and let $G$ be a solvable group with
metacyclic Sylow subgroups. Then $G$ is tamely $K$-admissible if
and only if its Sylow subgroups have Liedahl presentations.

\end{thm}


%

\begin{thm}\label{section2-sylow meta cyclic admissibility corollary}
Let $K$ be a number field. Let $G$ be a solvable group such that
the rational primes dividing $\card{G}$ do not decompose in $K$. Then
$G$ is $K$-admissible if and only if its Sylow subgroups are
metacyclic and have Liedahl presentations.
\end{thm}

In particular if a solvable group is tamely $M$-admissible then it
is also tamely $K$-admissible, i.e. tame admissibility has a going
down property for solvable groups. Also, if $G$ is solvable and
any prime $p\divides \card{G}$ satisfies Item (i) in Liedahl's
condition over $M$, i.e. does not decompose in $M$, then $G$ is
$M$-admissible if and only if $G$ is tamely $M$-admissible.

In particular for $M=\mQ$ one has that any solvable group $G$ that is $\mQ$-admissible is tamely $\mQ$-admissible. However over larger number fields this is no longer the case. Let us define wild $K$-admissibility:
%
%
\begin{defn}\label{section2- definition of wild admissiblity}
A $G$-extension $L/K$ is wildly $K$-adequate if $L/K$ is
$K$-adequate and there is a prime $p$ dividing $\card{G}$ such
that every prime $\nu$ of $K$ for which
$$\Gal(L_{\nu}/K_{\nu})\supseteq G(p),$$ divides $p$.
A $K$-admissible group $G$ is called wildly $K$-admissible if
every $K$-adequate $G$-extension is wildly $K$-adequate.
\end{defn}

Clearly a tamely $K$-admissible group is not wildly $K$-admissible. Theorems \ref{theorem tame admissibility of solvable groups} and \ref{section2-sylow meta cyclic admissibility corollary} guarantee
that a solvable group which is $K$-admissible but not wildly, is tamely
$K$-admissible. In particular:
\begin{rem}\label{section0 - necessary condition for wild admissibility}
Every $K$-admissible $p$-group which is not tamely $K$-admissible is wildly $K$-admissible. So, every non-metacyclic $K$-admissible $p$-group is wildly $K$-admissible.
\end{rem}

\subsection{The Grunwald-Neukirch (GN) property}\label{ss:GN}

A group $G$ {\bf has the  GN-property} (named after Grunwald and Neukirch)
over a number field $K$ if for every finite set $S$ of primes of
$K$ and corresponding subgroups $G^\nu\leq G$ for $\nu \in S$, there is
a Galois $G$-extension $L/K$ for which $\Gal(L_\nu/K_\nu)\cong G^\nu$
for every $\nu\in S$.

The Grunwald-Wang Theorem shows that except for special cases (see
\cite{Wan}), abelian groups $A$ have the GN-property over $K$. A
large set of examples comes from a Theorem of Neukirch
\cite[Corollary 2]{Neu}. Let $m(K)$ denote the number of roots of
unity in a number field $K$.
\begin{thm}[{Neukirch, \cite{Neu}}]\label{Neukirch main cor}
Let $K$ be a number field and $G$ a group for which $\card{G}$ is
prime to $m(K)$. Then $G$ has the GN-property over $K$.
\end{thm}

Another important source of examples is having a generic extension
(\cite[Theorem 5.9]{Sal3}):
\begin{thm}[Saltman]\label{generic ext}
If $G$ has a generic extension over a number field $K$ then
$G$ has the GN-property over $K$.\end{thm}

By \cite{Sal2}, if $\mu_p\subseteq K$ then any group of order $p^3$ which is not the cyclic group of order $8$ has a generic extension over $K$.  
In \cite{Sal1}, many groups are proved to have a generic extension over number fields, in particular, any abelian group that does not have an element of order $8$. In \cite{Sal1} it is also proved that the class of groups with a generic extension is closed under wreath products. In particular we have:
\begin{cor}[Saltman]\label{cor_sal} Let $q$ be an odd prime and let $K$ be a number field that contains the $q$-th roots of unity. Then any iterated wreath product of odd order cyclic groups and groups of order $q^3$ has the GN-property over $K$.      \end{cor}

For more examples see \cite{Nef}.
Under the assumption of the GN-property one has the following characterization of wild admissibility:
\begin{lem}\label{wildGN.lem} Let $K$ be a number field and $G$ a $K$-admissible group that has the GN-property over $K$. Then $G$ is wildly $K$-admissible if and only if there is a prime $p_0\divides \card{G}$ for which $G(p_0)$ does not have a Liedahl presentation over $K$.
\end{lem}
\begin{proof} Assume $p_0$ is a prime for which $G(p_0)$ does not have a Liedahl presentation over $K$. Assume on the contrary there is a $K$-adequate $G$-extension $L/K$ such that for every $p\divides \card{G}$ there is a prime $v$ of $K$ that does not divide $p$, with $\Gal(L_v/K_v)\supseteq G(p)$. Then $G(p_0)$ has a Liedahl presentation over $L^{G(p_0)}$ and  by \Rref{section0 - remark on down-inheritance of Lid cond} $G(p_0)$ has a Liedahl presentation over $K$, contradiction.

On the other hand if all Sylow subgroups have Liedahl presentations then by \Rref{interpert_lid.rem} every Sylow subgroup is realizable over infinitely many completions. One can therefore choose distinct primes $\{v_i(p)\subjectto i=1,2, p\divides \card{G}\}$ of $K$ such that $G(p)$ is realizable over $K_{v_i(p)}$ and $ v_i(p)\ndivides p$ for every $p\divides\card{G},i=1,2$. Since $G$ has the GN-property it follows that $G$ is tamely $K$-admissible.
\end{proof}
%
\subsection{Galois groups of local fields}\label{ss::localgal}

Let $k$ be a $p$-adic field of degree $n$ over $\mQ_p$. Let $q$ be the size of the residue field $\overline{k}$, and let $p^s$ be the size of the group of $p$-power roots of unity inside $\tame{k}$. Then
\begin{enumerate}
\item $\Gal(\unram{k}/k)$ is (topologically) generated by an automorphism $\s$, and isomorphic to $\hat{\Z}$;
\item $\Gal(\tame{k}/\unram{k})$ is (topologically) generated by an automorphism $\tau$, isomorphic to $\hat{\Z}^{(p')}$ (which is the complement of $\Z_p$ in $\hat{\Z}$);
\item The group $\Gal(\tame{k}/k)$ is a pro-finite group generated by $\sigma$ (lifting the above mentioned automorphism) and $\tau$, subject to the single relation $\s^{-1} \tau \s =  \tau^q$.
\end{enumerate}
Moreover, $\sigma$ and $\tau$ act on $\mu_{p^s}$ by exponentiation by some $g \in \Z_p$ and $h \in \Z_p$, respectively (Note that $g$ and $h$ are well defined modulo $p^s$).

Let $\overline{G_k(p)}$ denote the Galois group of the maximal $p$-extension of $k$ inside $\sep{k}$ (a separable closure of $k$), over $k$.
Let $s_0$ be the maximal number such that $k$ contains roots of unity of order $p^{s_0}$. Note that if $s_0 > 0$ then $n$ must be even. The following Theorem summarizes results of Shafarevich \cite{Shf}, Demushkin \cite{Dem}, Serre \cite{Ser} and Labute \cite{Lab}:
\begin{thm}[{\cite[Section II.5.6]{Serre}}]\label{SDL}
When $p^{s_0}\neq 2$, $\overline{G_k(p)}$ have the following presentation of pro-$p$
groups:
\begin{equation*}
\overline{G_k(p)}\cong  \left\{ \begin{array}{ll}
\langle x_1,\ldots,x_{n+2} \mid x_1^{p^{s_0}}[x_1,x_2] \cdots [x_{n+1},x_{n+2}]=1 \rangle, & \mbox{if $s_0 > 0$} \\
\langle x_1,\ldots,x_{n+1}  \rangle, & \mbox{if $s_0 = 0$}
\end{array} \right.;
\end{equation*}
When $p^{s_0} = 2$ and $n$ is odd,
$$\overline{G_k(p)}\cong \sg{x_1,\ldots,x_{n+2} \mid x_1^2 x_2^4 [x_2,x_3] \cdots [x_{n+1},x_{n+2}]=1 },$$
otherwise there is an $f\geq 2$ for which $\overline{G_k(p)}$ has one of the pro-$p$ presentations:
$$ \langle x_1,\ldots,x_{n+2} \mid x_1^{2}[x_1,x_2]x_3^{2^f}[x_3,x_4] \cdots [x_{n+1},x_{n+2}]=1 \rangle \mbox{\it , or} $$
$$\langle x_1,\ldots,x_{n+2} \mid x_1^{2+2^f}[x_1,x_2] \cdots [x_{n+1},x_{n+2}]=1 \rangle. $$
\end{thm}

\begin{thm}[Jannsen, Wingberg, {\cite{JW}, see also \cite[Theorem~7.5.10]{NSW}}]\label{JW}
The group $G_k$ has the following presentation (as a profinite group):
\begin{equation*}
G_k =  
\langle \sigma,\pro{\tau}{p'},\pro[N]{x_0,\ldots,x_{n}}{p} \subjectto \tau^\sigma = \tau^{q}, \,x_0^\sigma = \langle x_0,\tau \rangle^g x_1^{p^s}[x_1,x_2]\cdots [x_{n-1},x_{n}]\rangle 
\end{equation*} if $n$ is even, and
\begin{equation*}
G_k = \langle \sigma,\pro{\tau}{p'},\pro[N]{x_0,\ldots,x_{n}}{p} \subjectto \tau^\sigma =
\tau^{q}, \,x_0^\sigma = \langle x_0,\tau \rangle^{g}
x_1^{p^{s}}[x_1,y_1][x_2,x_3]\cdots [x_{n-1},x_{n}]\rangle 
\end{equation*}
if $n$ is odd, where $\pro{\tau}{p'}$ denotes that $\tau$ is a pro-$p'$ element (has order prime to $p$ in every finite quotient), and $\pro[N]{x_0,\dots,x_n}{p}$ denotes the condition that the closed normal subgroup generated by $x_0,...,x_n$ is required to be a pro-$p$ group.  Here, the closed subgroup generated by $\s$ and $\tau$ is isomorphic to $\Gal(\tame{k}/k)$. 
The notation $\langle x_0,\tau \rangle$ stands for $(x_0^{h^{p-1}} \tau x_0^{h^{p-2}} \tau \cdots x_0^h \tau)^{\frac{\pi_p}{p-1}}$, where $\pi_p \in \hat{\Z}$ is an element such that $\pi_p \hat{\Z} = \Z_p$. Also, $y_1$ is a multiple of $x_1^{\tau^{\pi_2(p+1)}}$ by an element in the maximal pro-$p$ quotient of the pro-finite group generated by $x_1,\sigma^{\pi_2}$ and $\tau^{\pi_2}$. In particular, in every pro-odd quotient of $G_k$, $[x_1,y_1]$ is trivial.
\end{thm}

\begin{rem}\label{local_finite_quotients.rem}
Notice that $G_k$ is a semidirect product of a pro-$p$ group $P_k$
and a profinite metacyclic group $D_k$, where $P_k$ is the closed
normal subgroup generated by $x_0,\dots,x_n$ and  $D_k$ is the
closed subgroup generated by $\sigma$ and $\tau$. The $p$-Sylow
subgroup of $G_k$ is therefore the pro-$p$ closure of
$\langle\sigma^{\pi_p}\rangle \cdot P_k$.
\end{rem}

\begin{rem}
If $G$ is admissible over a number field $K$, then for every $p$
there is a subgroup $H\supseteq G(p)$ which is realizable over a
completion of $K$. In particular, $H$ is a product of a metacyclic
group and a normal $p$-subgroup.
\end{rem}

The following result on realizability of metacyclic $p$-groups
will be used in \Sref{sec3}.
\begin{lem}\label{5.11}
Let $k$ be a $p$-adic field. Then any metacyclic $p$-group
$G$ is realizable over $k$.
\end{lem}
\begin{proof}
Let $G=\M(m,n,i,t)$ (see {\emph{\eq{Mdef}}}). The proof for $k \neq \mQ_2$ is in \cite{Nef2}. For $k = \mQ_2$ we
cover $2$-groups, so $m$ and $n$ are $2$-powers and $t$ is odd. In
this case $\overline{G_k(2)}$ has the pro-$2$ presentation
$\sg{a,b,c \subjectto a^2b^4[b,c] = 1}$ (by \Tref{SDL}), i.e.
$\overline{G_k(2)}$ is isomorphic to the free pro-$2$ group on
three generator modulo the normal closure of the single relation.
So the map $\phi: \overline{G_k(2)}\ra G$ defined by:
$$a \mapsto x^{-2} y^s,\ b \mapsto x,\ c \mapsto y,$$
is well defined (and surjective) whenever $(x^{-2}y^s)^2x^4[x,y]=1$.

As $t$ is odd, $t^2+1\equiv 2$ (mod $4$), $\frac{t^2+1}{2}$ is odd
and we can choose an $s \equiv
\frac{1}{\frac{t^2+1}{2}}\frac{1}{t^2}\frac{1-t}{2}$ (mod $n$) so
that $s(t^2+1)\equiv \frac{1-t}{t^2}$ (mod $n$). For such  $s$ one
has:
$$(x^{-2}y^s)^2x^4[x,y] = x^{-4}y^{st^{-2}+s}x^4y^{t-1}  =y^{s(t^{-2}+1)t^4+t-1}=1. $$ Thus $\phi$ is well defined.
\end{proof}

\section{$p$-groups}\label{sec2}

A nilpotent group $G$ is $K$-admissible if and only if all Sylow
subgroups of $G$ are $K$-admissible. In particular studying the
behavior of the admissibility of $G$ under extension of number
fields is reduced to understanding the behavior of its Sylow
subgroups.

\subsection{The case $p$ odd}\label{ss:podd}
We begin by proving the observation on realizability over
extensions of $p$-adic fields, $p$ odd.
\begin{proof}[Proof of \Pref{prop 1}]
Let $n$ denote the rank $\dimcol{k}{\mQ_p}$ and let
$t=\dimcol{m}{k}$. If $t=1$ there is nothing to prove. For $t=2$,
$(\dimcol{m}{k},p)=1$ and hence from a $G$-extension $l/k$ we can
form a $G$-extension $lm/m$. Now let $t>2$. 
It suffice to show that
$\overline{G_k(p)}$ is a quotient of $\overline{G_m(p)}$.

By \Tref{SDL}, $\overline{G_k(p)}$ and $\overline{G_m(p)}$ have the following presentations of pro-$p$
groups:
\begin{equation*}
\overline{G_k(p)}\cong  \left\{ \begin{array}{ll}
\langle x_1,\ldots,x_{n+2} \mid x_1^{p^{s_0}}[x_1,x_2] \cdots [x_{n+1},x_{n+2}]=1 \rangle, & \mbox{if $s_0 > 0$} \\
\langle x_1,\ldots,x_{n+1}  \rangle, & \mbox{if $s_0 = 0$}
\end{array} \right.,
\end{equation*}
and
\begin{equation*}
\overline{G_m(p)}\cong  \left\{ \begin{array}{ll}
\langle x_1,\ldots,x_{nt+2} \mid x_1^{p^{s_0'}}[x_1,x_2] \cdots [x_{nt+1},x_{nt+2}]=1 \rangle, & \mbox{if $s_0' > 0$} \\
\langle x_1,\ldots,x_{nt+1}  \rangle, & \mbox{if $s_0' = 0$}
\end{array} \right.,
\end{equation*}
where $p^{s_0}$ and $p^{s_0'}$ are the numbers of $p$-power roots
of unity in $k$ and $m$, respectively. Clearly $s_0 \leq s_0'$. Let $F_p(y_1,\ldots,y_k)$ denote the free
pro-$p$ group of rank $k$ with generators $y_1,\ldots,y_k$.  If
$s_0' = 0$ then we are done since $F_p(x_1,\ldots,x_{n+1})$ is a
quotient of $F_p({x_1,\ldots,x_{nt+1}})$.

Suppose $s_0' > 0$.
Let $\phi \co \overline{G_m(p)} \ra
F_p(y_1,\ldots,y_{\frac{nt+2}{2}})$ be the epimorphism defined by $\phi(x_{2i-1})=1$
and $\phi(x_{2i})=y_{i}$, $i = 1,\dots,\frac{nt+2}{2}$. Now as $t>2$ we have:
\begin{equation*} \frac{nt+2}{2}=n\frac{t}{2}+1\geq n+2
\end{equation*} and hence there is a projection $\pi$:
\begin{equation*}
\pi \co F_p(y_1,\ldots,y_{\frac{nt+2}{2}}) \ra \overline{G_k(p)}.
\end{equation*}
Thus $\pi\circ\phi \co \overline{G_m(p)}\ra \overline{G_k(p)}$ is an epimorphism. We deduce that every epimorphic image of $\overline{G_k(p)}$ is also an epimorphic image of $\overline{G_m(p)}$.
\end{proof}
%

We can now prove \Pref{intro-admiss of 2-gp}. It suffices
to prove:
\begin{prop}\label{odd-p-up.prop}
Let $M/K$ be an extension of number fields. Let $p$ be an odd
prime and $G$ a $p$-group that is $K$-admissible and has the
GN-property over $M$. If $G$ satisfies Liedahl's condition over
$M$, then $G$ is $M$-admissible.
\end{prop}
\begin{proof}
As $G$ is $K$-admissible, $G$ is realizable over $K_{v_1},K_{v_2}$
for two primes $v_1,v_2$ of $K$. We claim there are two primes
$w_1,w_2$ of $M$ for which $G$ is realizable over
$M_{w_1},M_{w_2}$. Since $G$ has the GN-property over $M$ proving
the claim shows $G$ is $M$-admissible. There are two cases:

Case I: $p$ decomposes in $M$.  If one of the primes $v_1,v_2$
does not divide $p$, then $G$ is metacyclic and hence by
\Lref{5.11}, $G$ is realizable over any $M_{w_1},M_{w_2}$ for any
two primes $w_1,w_2$ of $M$ that divide $p$. If on the other hand
both $v_1,v_2$ divide $p$ then by \Pref{prop 1}, $G$ is realizable
over $M_{w_i}$ for $w_i|v_i$, $i=1,2$,

Case II: $p$ does not decompose in $M$. Since $G$ satisfies
Liedahl's condition over $M$, $G$ has a Liedahl presentation over
$M$. In particular by \Tref{section2-sylow meta cyclic
admissibility corollary}, $G$ is $M$-admissible.
\end{proof}
As a corollary we deduce that wild admissibility goes up for
$p$-groups:
\begin{cor}\label{section2- prop on wild admissiblity of p-groups}
Let $p$ be an odd prime. Let $M/K$ be an extension of number
fields and $G$ a wildly $K$-admissible $p$-group that has the
GN-property over $M$. Then $G$ is also wildly $M$-admissible.
\end{cor}
\begin{proof}
Since $G$ is wildly $K$-admissible, $p$ decomposes in $K$ and
hence in $M$. Thus $G$ satisfies Liedahl's condition over $M$ and
by \Pref{odd-p-up.prop} $G$ is $M$-admissible. By remarks
\ref{interpert_lid.rem} and \ref{section0 - necessary condition for wild admissibility}, the
wild $K$-admissibility of $G$ implies that $G$ is not realizable
over $K_\nu$ for any prime $\nu$ which does not divide $p$. By
Remark \ref{section0 - remark on down-inheritance of Lid cond},
$G$ is also not realizable over $M_w$ for any prime $\varpi$ of
$M$ which does not divide $p$. Therefore an $M$-adequate
$G$-extension must also be wildly $M$-adequate and hence $G$ is
wildly $M$-admissible.
\end{proof}
Apply Theorem \ref{Neukirch main cor}, we have:
\begin{cor}
Let $p$ be an odd prime. Let $M/K$ be an extension of number
fields so that $M$ does not contain the $p$-th roots of unity. Let
$G$ be a $p$-group that is wildly $K$-admissible. Then $G$ is
wildly $M$-admissible.
\end{cor}

\subsection{The case $p = 2$}

As in \Pref{prop 1} we have:
\begin{lem}
Let $m/k$ is a finite extension of $2$-adic fields, which is
either of degree greater than $2$, or such that $m$ and $k$
contain $\sqrt{-1}$ and have the same $2$-power roots of unity.
Then any $2$-group realizable over $k$ is also realizable over $m$.
\end{lem}
\begin{proof}
If $\dimcol{m}{k} > 2$, the same proof as of \Pref{prop 1} holds
in all cases of \Tref{SDL}. If $\sqrt{-1}\in k,$ and $k$ and $m$
have the same number of $2$-power roots of unity then
$\oline{G_k(p)}$ and $\oline{G_m(p)}$ have the same type of
presentations in \Tref{SDL} and one can obtain an epimorphism:
$\oline{G_m(p)}\ra \oline{G_k(p)}$ simply by dividing by the
redundant generators of $\oline{G_m(p)}$. \end{proof}

However, we show that \Pref{prop 1} may fail for $p = 2$.
%
We begin with some group-theoretic preparations.

\begin{lem}\label{noQ}
The group \begin{eqnarray*} G & = & \left<a_1,a_2,a_3 \subjectto \mbox{$G'$ is central of exponent $2$}, \quad a_1^2 = [a_2,a_3],\
\,a_2^2 = a_3^2 = 1\right>.\end{eqnarray*}
is not a quotient of the pro-$p$ group
$$\Gamma = \sg{x_1,\dots,x_{4} \subjectto
x_1^{4}[x_1,x_2][x_3,x_4] = 1}.$$
\end{lem}
\begin{proof}
For $j, k = 1,2,3$, write $\alpha_{jk}=[a_j,a_k] \in G$.
Suppose $x_i \mapsto a_1^{t_{i,1}} a_2^{t_{i,2}} a_{3}^{t_{i,3}}z_i$ ($i = 1,\dots,4$) is an epimorphism $\Gamma \ra G$, where $z_i \in G'$. %
Then
$[x_{2i-1},x_{2i}] \mapsto \prod_{j,k = 1}^{3}
[a_j^{t_{2i-1,j}},a_k^{t_{2i,k}}] = \prod_{1 \leq k < j \leq 3}\alpha_{jk}^{t_{2i-1,j}t_{2i,k} -
t_{2i-1,k}t_{2i,j}}$. %
Since $\exp(G) = 4$, the defining relation of $\Gamma$ translates to
$$\prod_{i=1}^{2}\prod_{1 \leq k < j \leq 3} \alpha_{jk}^{t_{2i-1,j}t_{2i,k} - t_{2i-1,k}t_{2i,j}} = 1,$$
from which it follows that $t_{1,j}t_{2,k} -
t_{1,k}t_{2,j} + t_{3,j}t_{4,k} -
t_{3,k}t_{4,j} \equiv 0 \pmod{2}$ for every $1 \leq k < j
\leq 3$.

Let $V$ denote the vector space $\F_2^{4}$, endowed with the bilinear form $b \co V \times V \ra \F_2$ defined by
$b((v_i)_{i=1}^4,(v'_i)_{i=1}^4) = v_{1}v'_{2} - v_{2}v'_{1} + v_{3}v'_{4} - v_{4}v'_{3}$.
This is an alternating non-degenerate form (in fact, hyperbolic),
and letting $t^j \in V$ be the vectors $t^{j}_i = t_{i,j}$, we
have that $b(t^j,t^k) = 0$ for every $j,k = 1,2,3$. It follows that $T = \vspan\set{t^1,t^{2},t^{3}} \sub V$ is orthogonal to
itself. But then $\dim T \leq \frac{1}{2}\dim V = 2$, contradicting the assumption that  the induced map $\Gamma
\ra G/G'G^2 = C_2^3$ is surjective.
\end{proof}

\begin{cor}\label{G is not realizable over Q_2(i)}
There is a group of order $2^6$ which is realizable over $k = \mQ_2$ but not over $m = \mQ_2(\sqrt{-1})$.
\end{cor}
\begin{proof}
As before, we construct a quotient of $\oline{G_k(2)}$ which is not a quotient of $\oline{G_m(2)}$.
Let $G$ and $\Gamma$ be as in \Lref{noQ}. By \Tref{SDL}, $\oline{G_m(2)} \isom
\Gamma$ and $$\oline{G_k(2)} = \sg{x_1,x_2,x_3 \subjectto x_1^2x_2^4[x_2,x_3] = 1}.$$
Mapping $x_i \mapsto a_i$ projects $\oline{G_k(2)}$ onto $G$, which is not a quotient of $\Gamma$.
\end{proof}
It seems that $2^6$ is the minimal possible order for such a $2$-group.

\begin{rem}
Let $m/k$ be an extension of local fields. If there is one $2$-group which is realizable over $k$ but not over $m$, then there are infinitely many such groups. Indeed, let $G$ be such a group, and let $k'$ be a $G$-Galois extension of $k$; the Galois group of any $2$-extension of $k'$ which is Galois over $k$ has $G$ as a quotient, and so is not realizable over $m$.
\end{rem}

Let us apply this example to construct an extension of number fields
$M/K$ for which the group $G$ is wildly $K$-admissible but not $M$-admissible and not even $M$-preadmissible.

Let $p$ and $q$ be two primes for which:

1) $p\equiv 5$ (mod $8$)

2) $q\equiv 1$ (mod $8$)

3) $q$ is not a square mod $p$.

\begin{prop}\label{3.8}
Let $K=\mQ(\sqrt{q})$ and $M=K(i)$. Then $G$ is wildly $K$-admissible but not $M$-preadmissible.
\end{prop}
\begin{proof}
Since $G$ is a $2$-group that is not metacyclic it is realizable only over completions at primes dividing $2$. In particular if $G$ is $K$-admissible then $G$ is wildly $K$-admissible. As  $2$ splits in $K$, any (of the two) prime divisor $v$ of $2$ in $M$ has a completion $M_{v}\cong \mQ_2(i)$. By \Cref{G is not realizable over Q_2(i)},  $G$ is not realizable over $\mQ_2(i)$ and hence not $M$-preadmissible. It therefore remains to show $G$ is $K$-admissible.

The rational prime $p$ is inert in $K$. Let $\fp$ be the unique prime of $K$ that divides $p$. We have $N(\fp):=\card{\overline{K_\fp}}=p^2\equiv 1$ (mod $8$). Thus, $K_\fp$ has a totally ramified $C_8$-extension. Let $\frak{q}_1,\frak{q}_2$ be the two primes of $K$ dividing $2$.

Consider the field extension $L_0=K(\mu_8,\sqrt{p})/K$. It has a Galois group $$\Gal(L_0/K)\cong C_2^3\cong G/Z(G).$$ This extension is ramified only at $\frak{q}_1,\frak{q}_2$ and $\fp$. As $K_{\frak{q}_i}\cong \mQ_2$ and $\Gal((L_0)_{\frak{q}_i}/K_{\frak{q}_i})\cong C_2^3$, $(L_0)_{\frak{q}_i}/K_{\frak{q}_i}$ is the maximal abelian extension of $K_{\frak{q}_i}$ of exponent $2$, for $i=1,2$. Note that $N(\fp)\equiv 1$ (mod $8$) and hence $K_\fp$ contains $\mu_8$. 

Let us show that the central embedding problem
\begin{equation}\label{global embedding problem for 2-group}
\xymatrix{     & & & G_{K} \ar[d] \ar@{-->}[dl] \\
              0 \ar[r] & Z(G)\cong  C_2^3 \ar[r] & G \ar[r] & G/Z(G)\cong C_2^3 \ar[r] & 0, } \end{equation}

has a solution. Let $\pi$ denote the epimorphism $G\ra G/Z(G)$. By theorems 2.2 and 4.7 in \cite{Neu2}, there is a global solution to Problem \ref{global embedding problem for 2-group} if and only if there is a local solution at every prime of $K$. There is always a solution at primes of $K$ which are unramified in $L_0$ so it suffices to find solutions at $\fp,\frak{q}_1,\frak{q}_2$. Any $G$-extension of  $K_{\frak{q}_i}$ contains $(L_0)_{\frak{q}_i}$ (as it is the unique $C_2^3$ extensions of $\mQ_2$), $i=1,2$. Since $G$ is realizable over $\mQ_2$ we deduce that the induced local embedding problem
\begin{equation}\label{full local embedding problem for 2-group}
\xymatrix{       & G_{K_{\frak{q}_i}} \ar[d] \ar@{-->}[dl] & \\
                \pi^{-1}(\Gal((L_0)_{\frak{q}_i}/K_{\frak{q}_i}))=G \ar[r] & \Gal((L_0)_{\frak{q}_i}/K_{\frak{q}_i})\cong G/Z(G)\cong C_2^3 \ar[r] & 0, } \end{equation}
has a solution for $i=1,2$. Since $(L_0)_\fp$ is the ramified $C_2$-extension of $K_\fp$, it can be embedded into the totally ramified $C_4$-extension and hence the local embedding problem at $\fp$ has a solution.

Therefore, Embedding problem \ref{global embedding problem for 2-group} has a solution. Let $L$ be the corresponding solution field. As Problem \ref{global embedding problem for 2-group} is a Frattini embedding problem such a solution must be surjective globally and at $\{\frak{q}_1,\frak{q}_2\}$. Thus, $L_0/K$ can be embedded in a Galois $G$-extension $L/K$ for which $\Gal(L_{\frak{q}_i}/K_{\frak{q}_i})\cong G$, for $i=1,2$. The field $L$ is clearly $K$-adequate and hence $G$ is $K$-admissible. 
\end{proof}

\section{Realizability under extension of local fields}\label{sec:local}

Realizability of a group $G$ as a Galois group over a field $k$ is clearly a necessary condition for $k$-admissibility. When $k$ is a local field, the conditions are equivalent since a division algebra of index $n$ is split by every extension of degree $n$.

In this section we study realizability of groups under field extensions, assuming the fields are local.

\subsection{Totally ramified extensions}

We first note what happens under prime to $p$ local extensions:

\begin{lem}\label{section3 - lemma on prime to p extensions}
Let $G_1$  be a subgroup of $G$ that contains a $p$-Sylow subgroup of $G$ and is realizable over the $p$-adic field $k$. Let $m/k$ be a finite extension for which $(\dimcol{m}{k},p)=1$. Then there is a subgroup $G_2\leq G_1$ that contains a $p$-Sylow subgroup of $G$ and is realizable over $m$.
\end{lem}
\begin{proof}
Let $l/k$ be a $G_1$-extension. Then $lm/m$ is a Galois
extension with Galois group $G_2$ which is a subgroup of $G_1$ and
for which $[l:l\cap m]=\card{G_2}$. Since $(\dimcol{l\cap
m}{k},p)=1$, any $p^s\divides \dimcol{l}{k} = \card{G_1}$ also
divides $p^s \divides \dimcol{l}{l\cap m}=\card{G_2}$. Thus $G_2$
must also contain a $p$-Sylow subgroup of $G$.
\end{proof}

The case where $p$ divides the degree $\dimcol{m}{k}$ is more difficult. Let us consider next totally ramified extensions:
\begin{lem}\label{section3 - lemma on a ramified p-extension}
Let $p\not=2$. Let $G$ be a group, $k$ a $p$-adic field with $n=[k:\mQ_p]$ and $m/k$ a totally
ramified finite extension. Assume furthermore that $m/k$ is not the extension $\mQ_3(\zeta_9+\overline{\zeta}_9)/\mQ_3$. If $G$ is realizable over $k$ then $G$ is also realizable over $m$.
\end{lem}
\begin{rem}
This shows that if $G$ has a subgroup $G_1$ that contains a $p$-Sylow of $G$ and is realizable over $k$ then $G$ also has a subgroup $G_2$ that contains a $p$-Sylow of $G$ and is realizable over $m$ (moreover, $G_1$ is realizable over $m$).
\end{rem}
\begin{proof}
Let $m/k$ be a totally ramified extension of degree $r=\dimcol{m}{k}$.
We shall construct an epimorphism $G_m\ra G_k$. For this we shall consider the presentations given in \Tref{JW}. Denote the parameters of $k$ by $n,q,s,g$ and $h$. Then the degree of $m$ over $\mQ_p$ is $nr$ and its residue degree remains $q$. Denote the rest of the parameters over $m$ by $s',g'$ and $h'$ (the parameters that correspond to $s,g$ and $h$ in \Tref{JW}).
Then by \Tref{JW}, $G_m$ has the following presentation (as a profinite group):
\begin{equation*}
G_m =  
\langle \sigma,\pro{\tau}{p'},\pro[N]{x_0,\ldots,x_{nr}}{p} \subjectto \tau^\sigma = \tau^{q}, x_0^\sigma = \langle x_0,\tau \rangle^{g'} x_1^{p^{s'}}[x_1,x_2]\cdots [x_{nr-1},x_{nr}]\rangle, 
\end{equation*}
if $nr$ is even and  \begin{equation*}G_m = \langle \sigma,\pro{\tau}{p'},\pro[N]{x_0,\ldots,x_{nr}}{p} \subjectto \tau^\sigma =
\tau^{q}, x_0^\sigma = \langle x_0,\tau \rangle^{g'}
x_1^{p^{s'}}[x_1,y_1][x_2,x_3]\cdots [x_{nr-1},x_{nr}]\rangle,  
\end{equation*}
if $nr$ is odd. 
Let $P_k$ 
be the closed normal subgroup of $G_k$ 
 generated  by
$x_0,\ldots,x_n$ 
and let $D_k$ (resp. $D_m$) be the closed subgroup  generated by
$\sigma$ and $\tau$. By assumption, $P_k$ is a pro-$p$ group. Note that as $k$ and $m$ have the same residue degree (same $q$), $D_k\cong D_m$.


Let us construct the epimorphism from $G_m$ to $G_k$. First send $x_0$ and every $x_k$ with $k$ odd in the presentation of $G_m$ to $1$. We get an epimorphism
$$G_m \twoheadrightarrow \langle \sigma,\pro{\tau}{p'},\pro[N]{z_1,\dots,z_d}{p} \subjectto \sigma\tau\sigma^{-1}=\tau^q \rangle $$ 
where $d = \intpart{\frac{nr-1}{2}}$ and $\intpart{\gamma}$ denotes the smallest integer $\geq \gamma$. 
Let us continue under the assumption $d\geq n+1$. Then there is an epimorphism $F_p(d) \twoheadrightarrow F_p(n+1)$. We therefore obtain epimorphisms:
$$G_m \twoheadrightarrow \langle \sigma,\pro{\tau}{p'},\pro[N]{z_1,\dots,z_{n+1}}{p} \subjectto \sigma\tau\sigma^{-1}=\tau^q \rangle \twoheadrightarrow  D_k\ltimes P_k \twoheadrightarrow G_k.$$

The numerical condition $\intpart{\frac{nr-1}{2}}\geq n+1$ fails if and only if:
 \begin{enumerate}
 \item $r = 1$, or
 \item $r = 2$, or
 \item $r = 3$ and $n = 1$.
 \end{enumerate}
The case $r=1$ is trivial. Since $p$ is an odd prime, the cases $r = 2$ and $r = 3$ are done by Lemma \ref{section3 - lemma on prime to p extensions}, unless $p = 3$ and $\dimcol{m}{k} = r= 3$, in which case $n = 1$, so $k = \mQ_3$. If $m\not=\mQ_3(\zeta_9+\overline{\zeta}_9)$ then $m\cap \tame{k}=k$ and the parameters $g,h,s$ in the presentation of $G_{k}$ remain the same in the presentation of $G_{m}$. In such case there is an epimorphism from $G_m$ onto $G_{k}$ whose kernel is generated by $\langle x_2,x_3 \rangle$.
\end{proof}

So the case $k=\mQ_3$ and $m=\mQ_3(\zeta_9+\oline{\zeta}_9)$ remains open. This will be one of several \emph{\sensitive} cases.

\subsection{The sensitive cases}\label{ss:sc}

\begin{defn}
We call the extension $m/k$ \emph{\sensitive} if it is one of the following:
\begin{enumerate}
\item \label{specialiv} $k=\mQ_3$ and $m$ is the totally ramified $3$-extension $\mQ_3(\zeta_9+\oline{\zeta}_9)$,

\item \label{speciali}
 $k=\mQ_5$ and $m = \mQ_5(\zeta_{11})$ is the unramified $5$-extension,

\item \label{specialii} $\dimcol{k}{\mQ_3}=1,2,3$ and $m/k$ is the unramified $3$-extension,

\item \label{specialiii} $k=\mQ_3$ and $m = \mQ_3(\zeta_7)$ is the unramified $6$-extension.
\end{enumerate}
\end{defn}

\begin{rem}
There are $17$ sensitive field extensions, up to isomorphism: one over $\mQ_5$ and $16$ over $\mQ_3$. Fixing the algebraic closures of the respective $p$-adic fields, there is one sensitive $5$-adic extension and $27$ $3$-adic ones. This can be verified using the automated tools in \cite{jj}. We provide details in the Appendix.
\end{rem}

Let us formulate the problem in case (\ref{specialiv}) for odd order groups:

\begin{rem}
Given a field $F$ denote by $G_F^{\odd}$ the Galois group that corresponds to the maximal pro-odd Galois extension of $F$.
In the $p$-adic case, for odd $p$, this is obtained from the presentation of $G_k$ (see \Tref{JW}) simply by dividing by the $2$-part of $\sigma$ and $\tau$. In such case we get a presentation of $G_F^{\odd}$  by identifying $\sigma_2 = \tau_2=1$. We get that $y_1$ is a power of $x_1$ and hence $[x_1,y_1]=1$.
\end{rem}

\begin{ques}\label{odd_sensitive.que} Let $m/k$ be the {\sensitive} extension \eq{specialiv}.
Then $q = 3$; also $p^{s} = 3$ so we can choose $h = -1$.
For $m$ we have $p^{s_m} = 9$ and $\tau(\zeta_9+\zeta_9^{-1}) = \zeta_9+\zeta_9^{-1}$, so $h_m = -1$ as well.
Theorem \ref{JW} gives us the presentations:
\begin{equation*} G_k^{\odd} = \langle \sigma,\pro{\tau}{3'},\pro[N]{x_0,x_{1}}{3} \subjectto \tau^{\sigma}=\tau^3, x_0^{\sigma}=\langle x_0,\tau\rangle x_1^3 \rangle,
\end{equation*}  while:
\begin{equation*} G_m^{\odd} = \langle \sigma,\pro{\tau}{3'},\pro[N]{x_0,x_{1},x_2,x_3}{3} \subjectto \tau^{\sigma}=\tau^3, x_0^{\sigma}=\langle x_0,\tau\rangle x_1^9[x_2,x_3] \rangle,
\end{equation*} where $\sigma,\tau$ are of order prime to $2$ and $\langle x_0,\tau \rangle=(x_0\tau x_0^{-1}\tau )^{\frac{\pi}{2}}$, which has order a power of $3$ in every finite quotient. Does the following hold: Let $G$ be an epimorphic image of $G_k^{\odd}$, is there necessarily a subgroup $G(p)\leq G_0\leq G$ so that $G_0$ is an epimorphic image of $G_m^{\odd}$?
\end{ques}

Note that for a $3$-group $G$ the claim was proved in \Pref{prop 1}.

\begin{rem}
In fact quotients of $G_k^{\odd}$ with $\tau = 1$ can be covered:
the group $\sg{\sigma,\pro[N]{x_0,x_1}{3} \subjectto x_0^{\sigma}
= x_1^3} = \sg{\sigma,\pro[N]{x_1}{3}}$ is covered by $\sigma
\mapsto \sigma$, $\tau \mapsto 1$, $x_0 \mapsto 1$, $x_1 \mapsto
1$, $x_2 \mapsto x_1$ and $x_3 \mapsto 1$. This corresponds to
realization of $G$  over $k$ whose ramification index is a $3$-power. 
\end{rem}

\subsection{Extensions of local fields}\label{sec5}

We can now approach the general case. Recall the presentation of
$G_k$ from \Tref{JW}. Let $P_k$ denote the closed normal subgroup
generated by $x_0,\dots,x_n$ and $D_k$ the closed subgroup
generated by $\s$ and $\tau$, as in
\Rref{local_finite_quotients.rem}.
\begin{rem}\label{pre4.10}
\begin{enumerate}
\item Decompose $\langle \s \rangle $ into its $p$-primary part
generated by $\sigma_p$ and its complement generated by
$\sigma_{p'}$ so that $\s =\s_p\s_{p'}$, where $[\s_p,\s_{p'}] =
1$. Then the pro-$p$ closure of $\sg{\s_p}\cdot P_k$ is a $p$-Sylow subgroup of $G_k$.
\item\label{pre.2} In every finite quotient $\s_p$ is a power of $\s$, and so normalizes $\tau$. It follows that $[\tau,\s_p]$ is a power of $\tau$, and so a pro-$p'$ element.
\item The image of the closure of $\sg{\sigma_p}P_k$ is normal in a quotient of $G_k$ if and only if $\tau$ conjugates $\s_p$ into the closure of $\sg{\sigma_p}P_k$; but then the image of $[\tau,\s_p]$ is a pro-$p$ element, so by \eq{pre.2} this is the case if and only if the image of $[\tau,\s_p]$ is trivial.
\item\label{pre.4} Therefore, the maximal quotient of $G_k$ with a normal $p$-Sylow subgroup is defined by the relation $[\s_p,\tau] = 1$.
\end{enumerate}
\end{rem}

\begin{lem}\label{coverbar}
Let $p$ be an odd prime. Let $m/k$ be an extension of $p$-adic
fields with $f = \dimcol{\bar{m}}{\bar{k}}$ a $p$-power, and
$\intpart{\frac{nr-1}{2}} \geq n+2$ where $n = \dimcol{k}{\mQ_p}$
and $r = \dimcol{m}{k}$.

Then $G_m$ maps onto the maximal quotient $\bar{G}_k$ of $G_k$
with normal $p$-Sylow subgroup.
\end{lem}
\begin{proof}
We shall construct an epimorphism from $G_m$ to $\bar{G}_k$. Let
$s_m,g_m,h_m$ be the invariants $s,g,h$ in \Tref{JW} that correspond
to $m$, and let $n = \dimcol{k}{\mQ_p}$. \Tref{JW} gives the
following presentation of $G_m$:
\begin{equation*}
G_m =  
\langle \sigma,\pro{\tau}{p'},\pro[N]{x_0,\ldots,x_{nr}}{p} \subjectto \tau^\sigma = \tau^{q^{f}}, x_0^\sigma = \langle x_0,\tau \rangle^{g_m} x_1^{p^{s_m}}[x_1,x_2]\cdots [x_{nr-1},x_{nr}]\rangle, 
\end{equation*} if $nr$ is even and \begin{equation*}G_m = \langle \sigma,\pro{\tau}{p'},\pro[N]{x_0,\ldots,x_{nr}}{p} \subjectto \tau^\sigma =
\tau^{q^{f}},\end{equation*}\begin{equation*} x_0^\sigma = \langle x_0,\tau \rangle^{g_m}
x_1^{p^{s_m}}[x_1,y_1][x_2,x_3]\cdots [x_{nr-1},x_{nr}]\rangle,  
\end{equation*} if $nr$ is odd.

Let $P_k$ (resp. $P_m$) be
the closed normal subgroup generated by
$x_0,\ldots,x_{n}$ (resp. $x_0,\dots,x_{nr}$) in $G_k$ (resp. $G_m$) and let $D_k\leq G_k$ (resp. $D_m\leq G_m$) be the closed subgroup generated by
$\sigma,\tau$ in $G_k$ (resp. in $G_m$).

Set $d = \intpart{\frac{nr-1}{2}}$, so by assumption $d \geq n+2$. Similarly to \Lref{section3 - lemma on a ramified p-extension} (noting that this time $D_m$ can be viewed as a subgroup of index $f$ in $D_k$), we have an epimorphism
\begin{eqnarray*}
G_m & \twoheadrightarrow & \sg{\s, \pro{\tau}{p'}, \pro[N]{z_1,\dots,z_m}{p}  \subjectto \tau^\s = \tau^{q^f}} \\
& \cong & \sg{\pro{\s_p}{p},\pro{\s_{p'}}{p'}, \pro{\tau}{p'}, \pro[N]{z_1,\dots,z_m}{p}  \subjectto \tau^{\s_p \s_{p'}} = \tau^{q^f}, [\s_p,\s_{p'}] = 1}.
\end{eqnarray*}

Let us divide by the relations $z_m^f=\sigma_p$ and $\tau^{z_m}=\tau^{q\s_{p'}^{-1/f}}$, where $\s_{p'}^{-1/f}$ is well defined since $f$ is a $p$-power. We then obtain an epimorphism to
$$\langle \pro{\s_{p'}}{p'},\pro{\tau}{p'}, \pro[N]{z_1,\dots,z_m}{p} \subjectto \tau^{z_m}=\tau^{q\s_p'^{-1/f}},[z_m^f,\s_{p'}]=1 \rangle. $$

Adding the relation $[z_m,\s_{p'}]$ and sending $z_m \mapsto \s_p$ maps this group onto
$$\langle \pro{\s_{p'}}{p'},\pro{\tau}{p'}, \pro[N]{z_1,\dots,z_{m-1}, \s_p}{p} \subjectto \tau^{\s_p\s_{p'}^{1/f}}=\tau^{q},[\s_p,\s_{p'}]=1 \rangle. $$

Mapping $\s_{p'} \mapsto \s_{p'}^f$, this groups maps onto
$$\sg{\pro{\s_{p}}{p},\pro{\s_{p'}}{p'},\pro{\tau}{p'}, \pro[N]{x_0,\cdots,x_n}{p} \subjectto \tau^{\s_p\s_{p'}}=\tau^{q}, \, [\s_{p},\s_{p'}]= [\s_{p},\tau]=1},$$
since by \Rref{pre4.10} the assumption that the normal subgroup generated by $\s_p$
is a $p$-group is equivalent to $[\s_p,\tau] = 1$.

But $\bar{G}_k$ is a quotient of this group by \Tref{JW} and
\Rref{pre4.10}.\eq{pre.4}.
\end{proof}

Using this, we can now now prove:
\begin{proof}[Proof of \Tref{intro-local main thm}]
Let $n,q$ be as defined above for $k$, and let $r =
\dimcol{m}{k}$. Let $f=\dimcol{\oline{m}}{\oline{k}}=f_pf_{p'}$
where $f_p$ is a $p$-power and $f_{p'}$ is prime to $p$. There is
an unramified $C_f$-extension $m'/k$ which lies in $m$, and then
$m/m'$ is totally ramified. Denote by $m_p$ the subfield of $m'$
which is fixed by $C_{f_p}$. Let
$r'=\frac{r}{f_{p'}}=\dimcol{m}{m_p}$. By Lemma \ref{section3 -
lemma on prime to p extensions}, there is a subgroup $G_0\leq G$
that contains a $p$-Sylow subgroup of $G$ and an epimorphism
$\phi:G_{m_p}\ra G_0$. The list of sensitive cases satisfies that if
$m/k$ is non-sensitive and $m_p/k$ is unramified and prime to $p$,
then $m/m_p$ is also non-sensitive and therefore we can assume
without loss of generality that $m_p=k$, $G=G_0$, i.e $f_{p'}=1$,
$f=f_p$, and $r' = r$.

If $\intpart{\frac{nr-1}{2}} \geq n+2$ then $G$ is a quotient of $G_m$ by \Lref{coverbar}. This numerical condition fails if and only if
\begin{enumerate}
\item $r = 4,5$ and $n = 1$;
\item $r = 3$ and $n = 1,2,3$;
\item $r = 1,2$.
\end{enumerate}

The cases $r=1,2,4$ are covered by Lemma \ref{section3 - lemma on prime to p extensions}. We are left with cases $r=3,5$. For $r=5$, $n = 1$ so $k = \mQ_5$ and by \Lref{section3 - lemma on a ramified p-extension} we may assume $m/k$ is not totally ramified, so $m/k$ is the unramified $5$-extension which is sensitive.

Let $r=3$. Lemma \ref{section3 - lemma on prime to p extensions} covers the case $p\not=3$, so we may assume $p=3$. Note that $f \,|\, r$. If $f=1$, Lemma \ref{section3 - lemma on a ramified p-extension} applies, except for $m=\mQ_3(\zeta_9+\oline{\zeta_9})$ and $k=\mQ_3$, which is sensitive. If $f = 3$, then $m/k$ is the unramified $3$-extension and $n=[k:\mQ_3]=1,2,3$ which are all sensitive.
\end{proof}

The following example shows that the assumption in
\Tref{intro-local main thm} that the normal $p$-Sylow subgroup of
$G$ is normal, is essential.

\begin{exam}\label{local-counterexample.exm}
Let $p < q$ be odd primes such that $p^p \equiv 1 \pmod{q}$ and $p
\not\equiv 1 \pmod{q}$ (for example $p = 7$ and $q = 29$). Let
$G=C_p \wr D$ where $$D = \langle a,b\subjectto a^{p}=b^{q}=1,
a^{-1}ba=b^p\rangle.$$ Let $k=\mQ_p$ and $m/k$ the unramified
extension of degree $p$. Then $G$ is realizable over $k$ but there
is no subgroup of $G$ that contains a $p$-Sylow subgroup of $G$
and is realizable over $m$.
\end{exam}
\begin{proof}
Let $P=C_p^{pq}$, so that $G=P\rtimes D$. Then one has the
projections
\begin{equation*}
G_k\ra G_k^{\odd} = \langle
\sigma,\pro{\tau}{p'},\pro[N]{x_0,x_1}{p} \subjectto
\tau^{\sigma}=\tau^p, x_0^{\sigma}=\langle x_0,\tau\rangle x_1^p
\rangle \ra \end{equation*}\begin{equation*} \ra \langle
\sigma,\pro{\tau}{p'},\pro[N]{x_1}{p} \subjectto
\tau^{\sigma}=\tau^p,x_1^p=1\rangle,
\end{equation*}
where the latter two group are pro-odd and the second epimorphism
is obtained by dividing by $x_0$. The latter group maps onto $G$
by $\s \mapsto a$ and $\tau \mapsto b$. It is therefore left to
prove that for any homomorphism $\phi:G_m\ra G$, $\Im(\phi)$ does
not contain a $p$-Sylow subgroup of $G$. Assume on the contrary
that $H=\Im(\phi)$ does. Recall:
\begin{equation*}
G_m = \langle \sigma,\pro{\tau}{p'},\pro[N]{x_0,\ldots,x_{p}}{p}
\subjectto \tau^\sigma = \tau^{p^p}, x_0^\sigma = \langle x_0,\tau
\rangle x_1^{p}[x_1,y_1][x_2,x_3]\cdots [x_{p-1},x_{p}]\rangle,  
\end{equation*}

Since $q$ is the only prime dividing $\card{G}$ other than $p$, and $\tau$ is pro-$p'$, any map into $G$ must split through:
$$ G_m\ra \langle \sigma,\pro{\tau}{p'},\pro[N]{x_0,\ldots,x_{p}}{p} \subjectto [\sigma,\tau]=1,\tau^{q}=1, $$ $$ x_0^\sigma = \langle x_0,\tau \rangle
x_1^{p}[x_1,y_1][x_2,x_3]\cdots [x_{p-1},x_{p}]\rangle.$$
However the latter group has a normal $p$-Sylow subgroup which is the product of the closed normal subgroup generated by the $x_i$'s and the pro-$p$ group generated by $\sigma^{\pi_p}$.
In particular, letting $\pi\co G\ra G/P=D$ be the projection, the image of $\pi\phi$ has a normal $p$-Sylow subgroup. This implies $\pi\phi$ is not surjective. But $H$ contains a $p$-Sylow subgroup of $G$, so we must have $\Im(\pi\phi)=C_p$. Again since $H$ contains a $p$-Sylow subgroup, and in particular $P$, we must have $H=P\rtimes C$ where $C=C_p$ is a subgroup of $D$ and the action of $C$ on $P$ is induced from the action of $D$. Thus:
$$\rank(H)=\rank(H/[H,H])=\rank((P/[P,C])\times C)=q+1.$$
Since $H$ is a $p$-group any epimorphism to it must split through $\oline{G_m(p)}$. However $\rank(\oline{G_m(p)})=\dimcol{m}{k}+1=p+1$, leading to a contradiction.
\end{proof}

\section{Extensions of number fields}\label{sec3}

%

We shall now apply \Tref{intro-local main thm} to study admissibility and wild admissibility.
\subsection{Main Theorem}\label{ss:5.1}


\begin{proof}[Proof of \Tref{intro_main.thm}]
As mentioned in the introduction, Liedahl's condition is necessary. 
Let us show that if $G$ satisfies this condition then $G$ is $M$-admissible.

We claim that one can choose distinct primes $\varpi_i(p)$,
$i=1,2, p\divides\card{G}$, of $M$ and corresponding subgroups
$H_i(p)\leq G$ so that $H_i(p)$ contains a $p$-Sylow subgroup of
$G$ and is realizable over $M_{\varpi_i(p)}$, $i=1,2$. 

As $G$ is $K$-admissible, for
every $p\divides\card{G}$ there are two options:
\begin{enumerate}
\item \label{r1} there are two primes $\nu_1,\nu_2$ of $K$ dividing $p$ and two
subgroups $G(p)\leq G_i\leq G$ so that $G_i$ is realizable over
$K_{\nu_i}$, $i=1,2$.
\item \label{r2} $G(p)$ is realizable over $K_{v}$ for $\nu$ which does not divide $p$.
\end{enumerate}

In case \eq{r1} with $p$ odd and $G(p)$ normal in $G$, by
 \Tref{intro-local main thm}, for any prime
$\varpi$ dividing $\nu_1$ or $\nu_2$ there is a subgroup $G(p)\leq
H_\varpi \leq G$ that is realizable over $M_\varpi$. Choose two
such primes $\varpi_1(p),\varpi_2(p)$ and set $H_i(p) :=
H_{\varpi_i(p)}$ (the subgroups \Tref{intro-local main
thm} constructs). In case $G(p)$ is not normal or $p=2$, we
assumed $G(p)$ is metacyclic and by \Lref{5.11}, $G(p)$ is
realizable over any $M_\varpi$ for any prime  $\varpi$ dividing
$\nu_1$ or $\nu_2$. In such case similarly choose two such primes
$\varpi_1(p),\varpi_2(p)$ and set $H_i(p)= G(p)$.

In case \eq{r2}, $G(p)$ is metacyclic. If $p$ has more than one
prime divisor in $M$ then there are two primes
$\varpi_1(p),\varpi_2(p)$ so that $\varpi_i(p)$ divides
$p$ and by \Lref{5.11} $H_i(p):=G(p)$ is realizable over
$M_{\varpi_i(p)}$, $i=1,2$.

If $p$ has a unique prime divisor in $M$ then $G(p)$ is assumed to
have a Liedahl presentation. Liedahl's condition implies that
there are infinitely many primes $\varpi(p)$ for which $G(p)$ is
realizable over $M_{w(p)}$ (see Theorem~28 and Theorem~30 in
\cite{Lid2}). Thus, we can choose two primes
$\varpi_1(p),\varpi_2(p)$ for every prime $p\divides \card{G}$
that has only one prime divisor in $M$, so that the primes
$\varpi_i(p),i=1,2$, are not divisors of any prime
$q\divides\card{G}$ and are all distinct. For such $p$, we also
choose $H_i(p):=G(p)$.

We have covered all cases of behavior of divisors of rational
primes in $M$  and hence proved the claim. It follows that $G$ is
$M$-preadmissible and as $G$ has the GN-property over $M$, $G$ is
$M$-admissible. 
\end{proof}

\begin{rem}\label{sylowmetacyclic.rem} If $G$ has metacyclic Sylow subgroups, the proof of \Tref{intro_main.thm} does not use \Tref{intro-local main thm} and holds for sensitive extensions as well.
\end{rem}
\subsection{Wild admissibility}


%
%
%

As to wild admissibility \Tref{intro_main.thm} and
\Lref{wildGN.lem} give:
\begin{cor}
Let $M/K$ be a non-sensitive extension. Let $G$ be a
$K$-admissible group for which every Sylow subgroup is either
normal or metacyclic and the $2$-Sylow subgroups are metacyclic.
Assume $G$ has the GN-property over $M$, satisfies Liedahl's
condition over $M$ but there is a prime $p$ for which $G(p)$ does
not have a Liedahl presentation over $M$. Then $G$ is wildly
$M$-admissible.
\end{cor}
We deduce that for groups as in \Tref{intro_main.thm}, wild
admissibility goes up in the following sense that generalizes \Cref{section2- prop on wild admissiblity of p-groups}:
\begin{cor}
Let $M/K$ be a non-sensitive extension. Let $G$ be a wildly
$K$-admissible group for which every Sylow subgroup is either
normal or metacyclic and the $2$-Sylow subgroups are metacyclic. Assume $G$ has the GN-property over $K$ and
$M$ satisfies Liedahl's condition over $M$. Then $G$ is wildly
$M$-admissible.
\end{cor}
\begin{proof}
By \Tref{intro_main.thm}, $G$ is $M$-admissible. The assertion now
follows from \Lref{wildGN.lem}, applied to both $K$ and $M$, and
the fact that if $G(p)$ does not have a Liedahl presentation over
$K$ then $G(p)$ does not have a Liedahl presentation over $M$ (see
Remark \ref{section0 - remark on down-inheritance of Lid cond}).
\end{proof}

\subsection{Examples}

The following is an example in which
\Tref{intro_main.thm} is used to understand how admissibility behaves
under extensions of a given number field:
\begin{exam}\label{mod.exm}
Let $p,q$ be odd primes and $m$ an integer so that $m$ is not
square mod $q$ but is a square mod $p$ and $p\equiv q+1$ (mod
$q^2$). For example $p=13$, $q=3$, $m=14$. Let $K=\mQ(\sqrt{m})$
and $G=C_p\wr H$, where $H$ is one of the following groups:
\begin{enumerate}
\item\label{exm1}
$H= M_{q^3}$ is the modular group of order $q^3$, i.e. $$H=\langle
x,y| x^{-1}yx=y^{q+1}, x^q=y^{q^2}=1 \rangle.$$

\item\label{exm2}
$H=C_{pq}\times C_q$.

\item\label{exm3}
$H=C_t$ where $t\in \mathbb{N}_{\odd}$ is prime to $p$.
\end{enumerate}
We shall show in each of the cases $G$ is $K$-admissible. Let $M$ be any non-sensitive extension of $K$.

By \Tref{Neukirch main cor}, in case (\ref{exm1}) $G$ satisfies the
GN-property over any number field that does not have any $p$-th and $q$-th
roots of unity, in particular over $K$. By \Cref{cor_sal}, in
cases (\ref{exm2}),(\ref{exm3}), $G$ satisfies the GN-property over
any $M$ and in case (\ref{exm1}) if $M$
contains the $q$-th roots of unity.

In cases (\ref{exm2}),(\ref{exm3}), $G$ satisfies Liedahl's
condition over any $M$ and in case (\ref{exm1}), $G$ satisfies
Liedahl's condition over $M$ if and only if $q$ decomposes in $M$
or $M\cap \mQ(\mu_{q^2})\subseteq \mQ(\mu_q)$.

It follows from Theorem \ref{intro_main.thm}, that in cases
(\ref{exm2}) and (\ref{exm3}) $G$ is $M$-admissible. In case (\ref{exm1}), if one assumes $M$ does not contain any $p$-th and $q$-th roots of
unity or that $M$ contains the $q$-th roots of unity then $G$ is $M$-admissible if and only if $G$ satisfies Liedahl's condition.
\end{exam}

\begin{proof}
The prime $p$ splits (completely) in $K$. Denote it's prime divisors in $K$ by $v_1,v_2$. Then $K_{v_i}\cong \mQ_p$ for $i=1,2$.  Using the presentation of  $G_{\mQ_p}^{\odd}$ given in \Qref{odd_sensitive.que} and dividing by $x_0=1$ one obtains an epimorphism:
\begin{equation}\label{quotient_exm.equ} G_{\mQ_p}\ra \langle \sigma,\pro{\tau}{p'},\pro[N]{x_1}{p} \subjectto \tau^\sigma=\tau^p , x_1^3=1\rangle. \end{equation}
 Since $p\equiv q+1$ (mod $q^2$) there is an epimorphism $$\langle \sigma, \pro{\tau}{p'} \subjectto \tau^\sigma=\tau^p \rangle\ra M_{q^3}$$ which together with Epimorphism \ref{quotient_exm.equ} shows that  $C_p\wr M_{q^3}$ is an epimorphic image of $G_{\mQ_p}$.
The group $G$ in case \ref{exm3} can obtained as an epimorphic image of $G_{\mQ_p}$ after dividing \ref{quotient_exm.equ} by $\tau=1$. In case \ref{exm2}, since $q|p-1$, there is an epimorphism
$$\langle \sigma, \pro{\tau}{p'} \subjectto \tau^\sigma=\tau^p \rangle\ra C_{pq}\times C_q,$$ which together with \ref{quotient_exm.equ} can be used to construct an epimorphism onto $G$.

In particular $C_p\wr H$ is realizable over $K_{v_1},K_{v_2}$ in all cases. Since $M_{q^3}$, $C_q\times C_q$ and $C_t$ have Liedahl presentations over $K$, they are realizable over completions at infinitely many  primes of $K$. As $G$ has the GN-property over $K$, it follows that $G$ is $K$-admissible in all cases.
\end{proof}

\begin{rem}
As Case \ref{exm3} of \Eref{mod.exm}  shows, the rank of $p$-Sylow
subgroups of  $K$-admissible groups is not bounded as apposed to
the case of admissible $p$-group in which the rank of the group is bounded
(see \cite[Section 10]{Sch}). 
\end{rem}
\begin{rem}
Case \ref{exm2} in \Eref{mod.exm} is an example of a group for
which proving $M$-admissibility requires the use of all steps in
the proof of \Tref{intro-local main thm}.
\end{rem}

The following example shows that the assumption that every Sylow
subgroup is either normal or metacyclic is essential for
\Tref{intro_main.thm} even for odd order groups and non-sensitive
extensions:
\begin{exam}\label{global-counterexample}
As in \Eref{local-counterexample.exm}, let $p < q$ be odd primes
such that $p^p \equiv 1 \pmod{q}$ and $p \not\equiv 1 \pmod{q}$.
Let $G=C_p \wr D$ where $$D = \langle a,b\subjectto a^{p}=b^{q}=1,
a^{-1}ba=b^p\rangle.$$ Let $d$ be a non-square integer that is a
square mod $pq$ and $K=\mQ(\sqrt{d})$. Let $v_1,v_2$ be the primes of $K$ dividing $p$. Let $M/K$ be a $C_p$-extension in which
both $v_1$ and $v_2$ are inert and $M$ does not contain any
$p$-th and $q$-th roots of unity.

Since both $p$ and $q$ have more than one prime divisor in $M$,
$G$ satisfies Liedahl's condition over $M$. As $M$ does not have
any $p$-th and $q$-th roots of unity, by \Tref{Neukirch main cor}, $G$ has
the GN-property over $M$ and $K$. We shall now show $G$ is
$K$-admissible but not $M$-admissible. Note that the only
condition of \Tref{intro_main.thm} that fails is that either $G$
has a normal $p$-Sylow subgroup or a metacyclic one.
\end{exam}
\begin{proof}
By \Eref{local-counterexample.exm}, $G$ is realizable over $\mQ_p$
and hence over $K_{v_1},K_{v_2}$. As $G$ has the GN-property over
$K$, $G$ is $K$-admissible. On the other hand, since $G(p)$ is not
metacyclic, a subgroup of $G$ that contains $G(p)$ is realizable
only over completions of $M$ at prime divisors of $p$. Let $w_i$
be the prime dividing $v_i$ in $M$, $i=1,2$. Then $w_1,w_2$ are
the only primes dividing $p$ in $M$ but by
\Eref{local-counterexample.exm} $G$ is not realizable over
$M_{w_i}$, for $i=1,2$. In particular $G$ is not $M$-preadmissible
and not $M$-admissible.
\end{proof}

\section*{Appendix}

We use the Jannsen-Wingberg presentation of $G_{\mQ_3}$ to count the sensitive extensions, as defined in Subsection~\ref{ss:sc}, up to isomorphism.
\forget 
There are $209$ sensitive field extensions (NO). Indeed, there is a single extension in each of cases (1), (2) and (4). In case (3) the extension is unramified, so it suffices to count the ground fields, which we do by degrees over $\mQ_3$. In degree $1$ there is of course one case, and there are $\card{\md[2]{\mQ_3}} - 1 = 3$ quadratic extensions.

Since the abelianization of $\overline{G_{\mQ_3}(3)}$ is $C_3^2$, there are $\frac{3^2-1}{2} = 4$ Galois cubic extensions of $\mQ_3$. For every non-Galois cubic extension there is a unique cubic Galois extension of a quadratic extension of $\mQ_3$ (generated by the square root of the discriminant), and every such pair gives rise to $3$ non-Galois cubic extensions (since this is the number of involutions in $S_3$). Applying \Tref{SDL} to the three quadratic extensions, we get that the number of cubic Galois extensions is $\frac{3^4-1}{2} = 40$ for the unramified extension, and $\frac{3^3-1}{2} = 13$ for each of the two ramified ones. So we have $(13+13+40)\cdot 3 = 198$ non-Galois cubic extensions, and $1+1+(1+3+198)+1=209$ altogether.
\forgotten
There is a
single extension in each of cases (1), (2) and (4). In case (3)
the extension is unramified, so it suffices to count the ground
field $k$, which we do by degrees over $\mQ_3$. In degree
$1$ there is one case. In degree $2$ there are $\card{\md[2]{\mQ_3}} - 1 = 3$ quadratic extensions.

Since the abelianization of $\overline{G_{\mQ_3}(3)}$ is $C_3^2$,
there are $\frac{3^2-1}{2} = 4$ Galois cubic extensions of
$\mQ_3$. For every non-Galois cubic extension $k$ there is a
unique $S_3$-Galois extension of $\mQ_3$ (generated over $k$ by
the square root of the discriminant).

The $S_3$-Galois extensions of $\mQ_3$ are in one-to-one
correspondence to the normal subgroups of $G_{\mQ_3}$ with
quotient $S_3$. The number of such subgroups is the number of
epimorphisms from $G_{\mQ_3}$ to $S_3$, divided by
$\card{\Aut(S_3)} = 6$. When counting epimorphisms $\varphi \co
G_{\mQ_3} \ra S_3$, we may assume the generators are in $S_3$,
which simplifies the presentation a great deal. Since $p^s = q =
3$ and we may assume $g = 1$ and $h = -1$, the presentation is
$$G_{\mQ_3} = \langle \sigma,\pro{\tau}{3'},\pro[N]{x_0,x_{1}}{3} \subjectto \tau^\sigma =
\tau^{3}, \,x_0^\sigma = (x_0 \tau x_0^{-1} \tau)^{\frac{\pi}{2}}
x_1^{3}[x_1,y_1]\rangle.$$

However, since $x_1$ is a pro-$3$ element, we may assume $x_1^3 =
1$. Since all elements of order $3$ in $S_3$ commute with each
other, we may assume $[x_1,y_1] = 1$ (see \Tref{JW} for more
details on $y_1$). Since $\tau$ is a pro-$3'$ element,
$\varphi(\tau)$ has order at most $2$, so $\varphi(x_0 \tau
x_0^{-1} \tau)$ is a commutator, whose order must divide $3$.
Exponentiation by $\frac{\pi}{2}$ squares such elements. So every
epimorphism to $S_3$ splits through
$$\langle \sigma,\pro{\tau}{3'},\pro[N]{x_0,x_{1}}{3} \subjectto \tau^2 = x_0^3 = x_1^3 = 1, \, \tau^\sigma =
\tau^{3}, \,x_0^\sigma = (x_0 \tau x_0^{-1} \tau)^2 \rangle,$$ and
we count epimorphisms from this group. If $\varphi(\tau) = 1$ then
$\varphi(x_0) = 1$, so $\varphi(x_1)$ is a non-trivial element of
order $3$ and $\varphi(\sigma) \not \in \sg{\varphi(x_1)}$. There
are $6$ such epimorphisms. So assume $\varphi(\tau)$ is a
non-trivial involution. The relations give
$[\varphi(\sigma),\varphi(\tau)] = 1$, so $\varphi(\sigma)$ is
either $1$ or $\varphi(\tau)$. Moreover, it turns out that
$\varphi(x_0 \tau x_0^{-1} \tau)^2 = \varphi(x_0)$ whenever
$\varphi(x_0)$ has order dividing $3$. For each possible value of
$\varphi(\tau)$ we get $8$ epimorphisms with $\varphi(\s) = 1$ and
$2$ more with $\varphi(\s) = \varphi(\tau)$. There are $3$
involutions, providing us with $6 + 3 \cdot 10 = 36$ epimorphisms
all together. Dividing by the number of automorphisms, we have $6$
Galois extensions of $\mQ_3$ with Galois group $S_3$.

Each Galois extension of this type contains $3$ non-Galois cubic
extensions of $\mQ_3$.  
The $S_3$-extension is determined by the cubic extension,
being its Galois closure. So we have $3 \cdot 6 = 18$ non-Galois
cubic subfields of a fixed algebraic closure $\bar{\mQ_3}$, consisting of $6$ isomorphism classes.
Summing up, there are $1+1+(1+3+(4+6))+1=17$
sensitive field extensions, up to isomorphism.


\newcommand\arxiv[1]{{\texttt{#1}}}

\end{document}